\documentclass[11pt, a4paper]{amsart}
\usepackage[utf8]{inputenc}
\usepackage[margin=1in]{geometry}
\usepackage{amsmath, amssymb, amsthm}
\usepackage{mathrsfs}
\usepackage{graphicx}
\usepackage[colorlinks=true]{hyperref}
\hypersetup{urlcolor=blue, citecolor=red}
\usepackage[square,numbers]{natbib}

\newtheorem{theorem}{Theorem}[section]
\newtheorem{corollary}[theorem]{Corollary}
\newtheorem{lemma}[theorem]{Lemma}
\newtheorem{proposition}[theorem]{Proposition}

\theoremstyle{definition}
\newtheorem{definition}[theorem]{Definition}
\newtheorem{remark}[theorem]{Remark}
\newtheorem{example}[theorem]{Example}
\numberwithin{equation}{section}

\title{Weighted Fourier capacities and Frostman restriction estimates}

\author{Rodrigo León}
\address{Departamento de Matemática, Facultad de Ciencias Físicas y Matemáticas, Universidad de Concepción, Concepción, Chile.}
\email{rleon2020@udec.cl}

\author{Vicente Vergara}
\address{Departamento de Matemática, Facultad de Ciencias Físicas y Matemáticas, Universidad de Concepción, Concepción, Chile.}
\email{vvergaraa@udec.cl}


\date{\today}

\begin{document}

\begin{abstract}
We study a weighted Fourier capacity attached to compact sets,
\[
    c_{p,k}(K)=\inf\left\{\|k\widehat u\|_{L^p(\mathbb{R}^d)}:
    u\in C_0^\infty(\mathbb{R}^d),\ u\geq 1 \text{ on } K\right\},
\]
where $k$ is a positive tempered weight. The paper develops complementary upper and lower estimates for this quantity. The upper bounds are obtained from Frostman data, Fourier restriction estimates, and Bessel-potential majorants, which together provide admissible smooth majorants of the compact set. The lower bounds follow from a dual testing argument against positive measures supported on $K$, yielding estimates in terms of the weighted $L^{p'}$ size of their Fourier transforms. For sets of positive Lebesgue measure, testing with $\chi_K\,\mathrm{d}x$ connects the capacity with the weighted integrability of $\widehat{\chi_K}$. A Lebedev-type Sobolev criterion then gives non-triviality ranges governed by the upper Minkowski dimension of $\partial K$ and the growth of the weight. We also record consequences for Ahlfors-David regular sets and positive-measure Cantor-type examples.
\end{abstract}

\maketitle

\section{Introduction}
\label{sec:introduction}

Fourier-side quantities are a natural way to measure how the spatial size of a compact set interacts with frequency decay and fractional regularity. In classical potential theory, this interaction is often encoded through spatial kernels and energies. In the present paper we work directly on the Fourier side and study a weighted version of the Fourier capacity associated with a compact set.

Let $K\subset \mathbb{R}^d$ be compact, let $1\leq p\leq \infty$, and let $k$ be a positive tempered weight on $\mathbb{R}^d$. The main object is the number
\[
    c_{p,k}(K)
    =
    \inf\left\{
        \|k\widehat u\|_{L^p(\mathbb{R}^d)}:
        u\in C_0^\infty(\mathbb{R}^d),\
        u\geq 1 \text{ on } K
    \right\}.
\]
When $k\equiv 1$, this is the Fourier $p$-capacity considered in the line of questions raised by Maz'ya; see, for instance, \cite{maz2018seventy}. The presence of the weight allows the capacity to interact with prescribed frequency scales. For polynomial or Bessel-type weights this is closely related to fractional regularity on the Fourier side.

The paper develops two complementary mechanisms for estimating $c_{p,k}(K)$. The first one is an upper-bound mechanism based on Frostman measures. If a positive measure $\mu$ supported on $K$ has Frostman growth and Fourier decay, then restriction estimates of Mockenhaupt--Mitsis--Bak--Seeger type provide global integrability of $\widehat\mu$. Combined with Bessel potentials, mollification and a spatial cutoff, this produces admissible functions for $K$ and yields explicit upper bounds for $c_{p,k}(K)$. The resulting constants keep track of the weight, the restriction constant and the normalization needed to make the potential admissible.

The second mechanism is dual. If $\nu$ is a positive finite measure supported on $K$ and
\[
    k^{-1}(\cdot)\widehat\nu(-\cdot)\in L^{p'}(\mathbb{R}^d),
\]
then every admissible function can be tested against $\nu$, giving the lower bound
\[
    c_{p,k}(K)
    \geq
    \frac{\nu(K)}
    {\|k^{-1}(\cdot)\widehat\nu(-\cdot)\|_{L^{p'}(\mathbb{R}^d)}}.
\]
For sets of positive Lebesgue measure, a distinguished testing measure is $\nu=\chi_K\,\mathrm{d}x$. In that case the lower bound is governed by the weighted integrability of $\widehat{\chi_K}$. We use a Lebedev-type criterion to relate this integrability to the upper Minkowski dimension of $\partial K$.

In particular, if $|K|>0$, if $\alpha=\overline{\dim}_M(\partial K)$, and if the weight satisfies
\[
    \frac{1}{k(\xi)}\leq C(1+|\xi|^2)^{-\sigma/2},
\]
then the Lebedev mechanism gives non-triviality of $c_{p,k}(K)$ in the range
\[
    1\leq p'<2,
    \qquad
    p'>\frac{2d}{2d-\alpha+2\sigma},
\]
provided this range is non-empty. This shows explicitly how the weight compensates, or amplifies, the boundary roughness detected by the Fourier transform of $\chi_K$.

The main contributions of the paper can be summarized as follows. First, we formulate the weighted Fourier capacity in an operative framework adapted to compact sets and tempered weights. Second, we prove a dual lower bound against arbitrary positive testing measures supported on $K$. Third, using Frostman growth, Fourier decay and the Mockenhaupt--Mitsis--Bak--Seeger input, we obtain upper bounds by constructing admissible Bessel-potential majorants. Fourth, we combine the upper and lower mechanisms into bilateral estimates. Finally, we derive a weighted Lebedev lower bound for positive-measure sets and discuss Ahlfors-David and fat Cantor examples.

The paper is organized as follows. Section~\ref{sec:weighted-fourier-capacity} fixes the Fourier convention, the class of tempered weights, and the basic operative properties of $c_{p,k}(K)$. Section~\ref{sec:frostman-restriction} records the Frostman and Fourier-decay hypotheses and the restriction estimate used later. Section~\ref{sec:capacity-bounds} proves the upper and bilateral capacity estimates. Section~\ref{sec:lebedev-examples} develops the weighted Lebedev criterion and presents the examples.

\section{Weighted Fourier capacity and operational preliminaries}\label{sec:weighted-fourier-capacity}

This section fixes the notation used throughout the paper and records the basic facts about the number $c_{p,k}(K)$ used later. The capacity is treated as an operative Fourier quantity attached to $K$. Its usefulness comes from the fact that every admissible function can be tested against measures supported on $K$ and, in particular, against $\chi_K$ when $K$ has positive Lebesgue measure.

\subsection{Fourier convention and tempered weights}\label{subsec:fourier-convention-weights}

We write $\mathcal{S}(\mathbb{R}^d)$ for the Schwartz class, that is, the space of smooth functions $\varphi$ such that
\[
    \sup_{x\in\mathbb{R}^d}
    |x^\alpha \partial^\beta \varphi(x)|<\infty
\]
for every pair of multi-indices $\alpha,\beta$. We also write
$C_0^\infty(\mathbb{R}^d)$ for the space of smooth compactly supported functions. With the Fourier convention \eqref{eq:fourier-convention}, the Fourier transform maps $\mathcal{S}(\mathbb{R}^d)$ onto itself, and every function in $C_0^\infty(\mathbb{R}^d)$ belongs to $\mathcal{S}(\mathbb{R}^d)$. Hence, if $u\in C_0^\infty(\mathbb{R}^d)$, then $\widehat u$ is rapidly decreasing. We use these standard facts throughout; see, for instance, \cite[Chapter 5]{stein2011fourier} or \cite[Chapter 8]{folland2009fourier}.

For $f\in L^1(\mathbb{R}^d)$ we use the Fourier transform
\begin{equation}\label{eq:fourier-convention}
    \widehat f(\xi)=\int_{\mathbb{R}^d} f(x)e^{-2\pi i x\cdot \xi}\,\mathrm{d}x,
    \qquad \xi\in \mathbb{R}^d,
\end{equation}
extended to tempered distributions in the usual way. If $1\leq p\leq \infty$, the conjugate exponent is denoted by $p'$, with the usual convention $1/\infty=0$.

For a bounded set $F\subset\mathbb{R}^d$ and $\delta>0$, we write
\[
    F_\delta
    :=
    \{x\in\mathbb{R}^d:\operatorname{dist}(x,F)<\delta\}.
\]
The upper Minkowski dimension of $F$ is
\[
    \overline{\dim}_{M}F
    :=
    \inf\left\{
    0\leq \alpha\leq d:
    |F_\delta|=O(\delta^{d-\alpha})
    \text{ as }\delta\downarrow0
    \right\}.
\]
Equivalently, $\overline{\dim}_{M}F\leq\alpha$ means that, for every
$\varepsilon>0$, there are constants $C_{\varepsilon,F}>0$ and
$\delta_{\varepsilon,F}>0$ such that
\[
    |F_\delta|
    \leq
    C_{\varepsilon,F}\delta^{d-\alpha-\varepsilon},
    \qquad 0<\delta<\delta_{\varepsilon,F}.
\]

A measurable function $k:\mathbb{R}^d\to (0,\infty)$ is called a tempered weight if there are constants $C_k,N_k>0$ such that
\begin{equation}\label{eq:tempered-weight}
    k(\xi+\eta)\leq C_k(1+|\xi|)^{N_k}k(\eta),
    \qquad \xi,\eta\in\mathbb{R}^d.
\end{equation}
Equivalently, after changing the constants, translations of $k$ are controlled by polynomial factors. We shall use the associated moderation function
\begin{equation}\label{eq:moderation-function}
    M_k(\xi):=\sup_{\eta\in\mathbb{R}^d}\frac{k(\xi+\eta)}{k(\eta)}.
\end{equation}
By \eqref{eq:tempered-weight}, $M_k(\xi)\leq C_k(1+|\xi|)^{N_k}$. Typical examples are the polynomial weights
\begin{equation}\label{eq:polynomial-weight}
    k_s(\xi)=(1+|\xi|^2)^{s/2},
    \qquad s\in\mathbb{R}.
\end{equation}

For $1\leq p\leq \infty$ and a tempered weight $k$, set
\begin{equation}\label{eq:weighted-fourier-seminorm}
    \|u\|_{p,k}:=\|k\widehat u\|_{L^p(\mathbb{R}^d)},
\end{equation}
whenever the right-hand side is meaningful and finite. In particular, every $u\in C_0^\infty(\mathbb{R}^d)$ satisfies
$\|u\|_{p,k}<\infty$. We shall use the standard multiplier estimate for H\"ormander-type spaces: if $u$ is such that $\|u\|_{p,k}<\infty$ and $\phi\in\mathcal{S}(\mathbb{R}^d)$, then
\begin{equation}\label{eq:schwartz-multiplier-estimate}
    \|\phi u\|_{p,k}
    \leq \|M_k\widehat\phi\|_{L^1(\mathbb{R}^d)}\|u\|_{p,k}.
\end{equation}
This is just Young's inequality applied to
\[
    k(\xi)\widehat{\phi u}(\xi)
    =k(\xi)\int_{\mathbb{R}^d}\widehat\phi(\xi-\eta)\widehat u(\eta)\,\mathrm{d}\eta
\]
and the definition of $M_k$. The quantity $\|M_k\widehat\phi\|_{L^1}$ is finite because $M_k$ has polynomial growth and $\widehat\phi$ is rapidly decreasing.

\subsection{The number $c_{p,k}(K)$}\label{subsec:capacity-definition}

\begin{definition}\label{def:weighted-fourier-capacity}
Let $K\subset\mathbb{R}^d$ be compact, let $1\leq p\leq \infty$, and let $k$ be a tempered weight. The weighted Fourier $(p,k)$-capacity of $K$ is
\begin{equation}\label{eq:weighted-fourier-capacity}
    c_{p,k}(K):=
    \inf\left\{
    \|u\|_{p,k}:
    u\in C_0^\infty(\mathbb{R}^d),\ u\geq 1\text{ on }K
    \right\}.
\end{equation}
We will call a function $u\in C^\infty_0(\mathbb{R}^d)$ an ``admissible function" for the compact set $K$ if it satisfies the condition $u\ge1$ on $K$. The admissibility condition is pointwise and real-valued: an admissible test function may be taken real-valued, but it is not required to be non-negative away from $K$.
\end{definition}

The following elementary facts are the only structural properties of $c_{p,k}$ used in the sequel.

\begin{lemma}[Finiteness and monotonicity]\label{lem:capacity-finiteness-monotonicity}
Let $1\leq p\leq \infty$ and let $k$ be a tempered weight.
\begin{enumerate}
    \item If $K\subset\mathbb{R}^d$ is compact, then $c_{p,k}(K)<\infty$.
    \item If $K_1\subset K_2$ are compact, then
    \[
        c_{p,k}(K_1)\leq c_{p,k}(K_2).
    \]
    \item $c_{p,k}(\emptyset)=0$.
\end{enumerate}
\end{lemma}

\begin{proof}
We first prove finiteness. Since $K$ is compact, there exists $R>0$ such that
\[
    K\subset B(0,R).
\]
Choose $\phi\in C_0^\infty(\mathbb{R}^d)$ such that
\[
    \phi(x)=1 \quad \text{for } x\in B(0,R).
\]
Then $\phi$ is admissible for $K$. It remains to check that
$\|\phi\|_{p,k}<\infty$.

Since $\phi\in C_0^\infty(\mathbb{R}^d)$, we have
$\widehat\phi\in\mathcal{S}(\mathbb{R}^d)$. In particular, for every
$N>0$ there is a constant $C_N$ such that
\[
    |\widehat\phi(\xi)|\leq C_N(1+|\xi|)^{-N},
    \qquad \xi\in\mathbb{R}^d.
\]
On the other hand, applying \eqref{eq:tempered-weight} with $\eta=0$
gives
\[
    k(\xi)\leq C_k k(0)(1+|\xi|)^{N_k}.
\]
Therefore
\[
    k(\xi)|\widehat\phi(\xi)|
    \leq C\, (1+|\xi|)^{N_k-N}.
\]
Choosing $N$ sufficiently large gives
$k\widehat\phi\in L^p(\mathbb{R}^d)$, with the usual interpretation
when $p=\infty$. Thus
\[
    c_{p,k}(K)\leq \|\phi\|_{p,k}<\infty.
\]

We next prove monotonicity. If $K_1\subset K_2$, then every admissible
test function for $K_2$ is also admissible for $K_1$. Hence the infimum
over admissible functions for $K_1$ is taken over a larger class, and
therefore
\[
    c_{p,k}(K_1)\leq c_{p,k}(K_2).
\]

Finally, every function in $C_0^\infty(\mathbb{R}^d)$ is admissible for
the empty set. Fix any $\phi\in C_0^\infty(\mathbb{R}^d)$. For every
$\varepsilon>0$, the function $\varepsilon\phi$ is admissible for
$\emptyset$, and by homogeneity,
\[
    c_{p,k}(\emptyset)\leq \varepsilon\|\phi\|_{p,k}.
\]
Letting $\varepsilon\downarrow0$ gives $c_{p,k}(\emptyset)=0$.
\end{proof}

\begin{lemma}[Outer stability along compact supersets]\label{lem:capacity-outer-stability-compact}
Let $K\subset\mathbb{R}^d$ be compact. If $a>c_{p,k}(K)$, then there exists an open set $U\supset K$ such that
\[
    c_{p,k}(C)<a
\]
for every compact set $C\subset U$.
\end{lemma}

\begin{proof}
Choose $u\in C_0^\infty(\mathbb{R}^d)$ admissible for $K$ with $\|u\|_{p,k}<a$. Pick $\delta\in(0,1)$ so small that $(1-\delta)^{-1}\|u\|_{p,k}<a$, and set
\[
    U:=\{x\in\mathbb{R}^d:u(x)>1-\delta\}.
\]
Then $U$ is open and contains $K$. If $C\subset U$ is compact, the function $(1-\delta)^{-1}u$ is admissible for $C$, and therefore
\[
    c_{p,k}(C)
    \leq (1-\delta)^{-1}\|u\|_{p,k}<a.
\]
\end{proof}

\begin{remark}\label{rem:localization-and-sums}
The triangle inequality gives
\[
    \|u_1+u_2\|_{p,k}\leq \|u_1\|_{p,k}+\|u_2\|_{p,k}
\]
for admissible test functions. When sums or localizations are used below, the corresponding admissibility and Fourier-side bounds are checked at the point of use. In particular, multiplication by a Schwartz cutoff is controlled by \eqref{eq:schwartz-multiplier-estimate}, with constants determined by the cutoff and the moderation function $M_k$.
\end{remark}

\subsection{Testing admissible functions against measures}\label{subsec:testing-against-measures}

The basic bridge between $c_{p,k}$ and measures supported on $K$ is duality. This is the mechanism by which Frostman measures enter the paper.

\begin{proposition}[Dual lower bound]\label{prop:capacity-dual-lower-bound}
Let $1\leq p\leq \infty$, let $k$ be a tempered weight, and let $K\subset\mathbb{R}^d$ be compact. Let $\mu$ be a finite positive Borel measure supported on $K$. Assume that
\begin{equation}\label{eq:dual-measure-integrability}
    k^{-1}(\cdot)\widehat\mu(-\cdot)\in L^{p'}(\mathbb{R}^d).
\end{equation}
Then
\begin{equation}\label{eq:capacity-dual-lower-bound}
    c_{p,k}(K)
    \geq
    \frac{\mu(K)}{\|k^{-1}(\cdot)\widehat\mu(-\cdot)\|_{L^{p'}(\mathbb{R}^d)}}.
\end{equation}
Here the quotient is interpreted as $0$ if the denominator is infinite.
\end{proposition}

\begin{proof}
Let $u\in C_0^\infty(\mathbb{R}^d)$ be admissible for $K$. Since $\mu$ is positive and supported on $K$,
\[
    \mu(K)\leq \int_K u(x)\,\mathrm{d}\mu(x)
    =\int_{\mathbb{R}^d}u(x)\,\mathrm{d}\mu(x).
\]
By Fourier inversion in the distributional pairing,
\[
    \int_{\mathbb{R}^d}u(x)\,\mathrm{d}\mu(x)
    =\int_{\mathbb{R}^d}\widehat u(\xi)\widehat\mu(-\xi)\,\mathrm{d}\xi.
\]
Therefore, by H\"older's inequality,
\[
    \mu(K)
    \leq
    \|k\widehat u\|_{L^p(\mathbb{R}^d)}
    \|k^{-1}(\cdot)\widehat\mu(-\cdot)\|_{L^{p'}(\mathbb{R}^d)}.
\]
Taking the infimum over admissible $u$ proves \eqref{eq:capacity-dual-lower-bound}.
\end{proof}

\begin{corollary}[Testing with Frostman measures]\label{cor:frostman-measure-lower-bound}
Let $\mu$ be a Frostman measure supported on a compact set $K$, normalized so that $\mu(K)>0$. If \eqref{eq:dual-measure-integrability} holds, then $c_{p,k}(K)>0$ and the quantitative lower bound \eqref{eq:capacity-dual-lower-bound} holds.
\end{corollary}

\begin{proof}
This is Proposition \ref{prop:capacity-dual-lower-bound} applied to the positive measure $\mu$.
\end{proof}

\begin{corollary}[Testing with $\chi_K$]\label{cor:indicator-lower-bound}
Let $K\subset\mathbb{R}^d$ be compact with positive Lebesgue measure. If
\begin{equation}\label{eq:indicator-integrability-condition}
    k^{-1}(\cdot)\widehat{\chi_K}(-\cdot)
    \in L^{p'}(\mathbb{R}^d),
\end{equation}
then
\begin{equation}\label{eq:indicator-lower-bound}
    c_{p,k}(K)
    \geq
    \frac{|K|}{\|k^{-1}(\cdot)\widehat{\chi_K}(-\cdot)\|_{L^{p'}(\mathbb{R}^d)}}.
\end{equation}
In particular, under \eqref{eq:indicator-integrability-condition}, $c_{p,k}(K)>0$.
\end{corollary}

\begin{proof}
Apply Proposition \ref{prop:capacity-dual-lower-bound} to the measure $\mathrm{d}\mu=\chi_K\,\mathrm{d}x$.
\end{proof}

\begin{remark}\label{rem:role-of-frostman-and-indicator}
The lower bounds above explain why the number $c_{p,k}(K)$ is naturally linked both to Frostman measures and to $\widehat{\chi_K}$. Frostman measures are useful for singular sets because their spatial growth and Fourier decay can imply \eqref{eq:dual-measure-integrability}. The special choice $\mu=\chi_K\,\mathrm{d}x$ is available only when $|K|>0$ and leads to estimates governed by the weighted integrability of $\widehat{\chi_K}$. The next sections quantify these two mechanisms.
\end{remark}

\section{Frostman measures and restriction estimates}
\label{sec:frostman-restriction}

This section records the geometric and Fourier-analytic hypotheses on measures supported on a compact set $K\subset\mathbb{R}^d$. These hypotheses provide the quantitative input used later to estimate the weighted Fourier capacity $c_{p,k}(K)$.

\subsection{Frostman growth}
\label{subsec:frostman-growth}

\begin{definition}[Frostman growth]
\label{def:frostman-growth}
Let $K\subset\mathbb{R}^d$ be compact and let $0<\gamma\leq d$. A finite positive Borel measure $\mu$ supported on $K$ has Frostman growth of exponent $\gamma$ if there exists a constant $C_\mu>0$ such that
\begin{equation}
\label{eq:frostman-growth}
    \mu(B(x,r))\leq C_\mu r^\gamma,
    \qquad x\in\mathbb{R}^d,\quad r>0.
\end{equation}
Let $[\mu]_{\gamma}$ be the smallest constant such that \eqref{eq:frostman-growth} holds.
\end{definition}

The Frostman condition is the measure-theoretic form of a lower dimensional size assumption on $K$. In particular, Frostman's lemma gives such measures when the Hausdorff content of $K$ is positive at the corresponding scale; see, for instance, \cite[Chapter 8]{mattila1999geometry}.

\subsection{Fourier decay of measures}
\label{subsec:fourier-decay-measures}

\begin{definition}[Fourier decay]
\label{def:fourier-decay-measure}
Let $K\subset\mathbb{R}^d$ be compact and let $0<\beta\leq d$. A finite positive Borel measure $\mu$ supported on $K$ has Fourier decay of order $\beta$ if there exists a constant $C_\beta>0$ such that
\begin{equation}
\label{eq:fourier-decay-measure}
    |\widehat\mu(\xi)|\leq C_\beta (1+|\xi|)^{-\beta/2},
    \qquad \xi\in\mathbb{R}^d.
\end{equation}
The least admissible constant in \eqref{eq:fourier-decay-measure} is denoted by $[\mu]_{\widehat{\ },\beta}$.
\end{definition}

The estimate \eqref{eq:fourier-decay-measure} is a property of the chosen measure. When such a measure exists for some exponent $\beta$, the support has Fourier dimension at least $\beta$ in the usual sense.

\subsection{The Mockenhaupt--Mitsis--Bak--Seeger input}
\label{subsec:mmbs-input}

The following restriction estimate is the black-box harmonic-analytic input used in the sequel. We state it in the form needed below; the exponent is the standard Mockenhaupt--Mitsis--Bak--Seeger threshold. See \cite{fraser2025sharpness,mockenhaupt-sharpness} for this formulation and sharpness results.

\begin{theorem}[Mockenhaupt--Mitsis--Bak--Seeger]
\label{thm:mmbs-restriction}
Let $0<\gamma,\beta<d$, and let $\mu$ be a finite positive Borel measure on $\mathbb{R}^d$ satisfying the Frostman growth condition \eqref{eq:frostman-growth} and the Fourier decay condition \eqref{eq:fourier-decay-measure}. Define
\begin{equation}
\label{eq:mmbs-threshold}
    p_\ast(\gamma,\beta,d)
    :=
    \frac{4d-4\gamma+2\beta}{\beta}.
\end{equation}
Then, for every $p\geq p_\ast(\gamma,\beta,d)$, there exists a constant
\[
    C_{\mathrm{MMBS}}
    =
    C_{\mathrm{MMBS}}(d,p,\gamma,\beta,[\mu]_{\gamma},[\mu]_{\widehat{\ },\beta})
\]
such that
\begin{equation}
\label{eq:mmbs-extension-estimate}
    \|\widehat{f\,\mathrm d\mu}\|_{L^p(\mathbb{R}^d)}
    \leq
    C_{\mathrm{MMBS}}\|f\|_{L^2(\mu)}
\end{equation}
for every $f\in L^2(\mu)$.
\end{theorem}

\begin{corollary}[Integrability of $\widehat\mu$]
\label{cor:mmbs-integrability-measure}
Under the hypotheses of Theorem \ref{thm:mmbs-restriction}, for every $p\geq p_\ast(\gamma,\beta,d)$,
\begin{equation}
\label{eq:mmbs-measure-integrability}
    \|\widehat\mu\|_{L^p(\mathbb{R}^d)}
    \leq
    C_{\mathrm{MMBS}}\mu(K)^{1/2},
\end{equation}
where $K=\operatorname{supp}\mu$ and $C_{\mathrm{MMBS}}$ is the constant from \eqref{eq:mmbs-extension-estimate}.
\end{corollary}

\begin{proof}
Apply \eqref{eq:mmbs-extension-estimate} with $f\equiv 1$ on $K$. Since
\[
    \|1\|_{L^2(\mu)}=\mu(K)^{1/2},
\]
and $\widehat{1\,\mathrm d\mu}=\widehat\mu$, the estimate \eqref{eq:mmbs-measure-integrability} follows.
\end{proof}

\section{Capacity bounds from Frostman data}
\label{sec:capacity-bounds}

We now use the measure-theoretic input from Section \ref{sec:frostman-restriction} to produce admissible functions for the weighted Fourier capacity. The construction is based on Bessel potentials, followed by a smooth compactly supported localization. The lower estimates are those obtained by dual testing in Proposition \ref{prop:capacity-dual-lower-bound}; the present section combines that dual mechanism with upper bounds coming from Frostman--MMBS data.

\subsection{A Bessel-potential admissible majorant}
\label{subsec:bessel-majorant}

For $\alpha\geq 0$, let $G_\alpha$ denote the Bessel kernel associated with the Fourier convention \eqref{eq:fourier-convention}, normalized by
\begin{equation}
\label{eq:bessel-kernel-fourier}
    \widehat{G_\alpha}(\xi)=(1+|\xi|^2)^{-\alpha/2}.
\end{equation}
The kernel $G_\alpha$ is positive and locally integrable; see, for instance, \cite[Chapter V]{stein1970singular}. We shall use the following elementary consequence of positivity and compactness.

\begin{lemma}[Bessel-potential majorant]
\label{lem:bessel-potential-majorant}
Let $K\subset\mathbb{R}^d$ be compact, let $\mu$ be a finite positive nonzero Borel measure supported on $K$, and let $\alpha\geq 0$. Let $\rho\in C_0^\infty(\mathbb{R}^d)$ be nonnegative, with $\int_{\mathbb{R}^d}\rho(x)\,\mathrm{d}x=1$, and set
\[
    \rho_\varepsilon(x):=\varepsilon^{-d}\rho(x/\varepsilon),
    \qquad \varepsilon>0.
\]
Define
\begin{equation}
\label{eq:mollified-bessel-potential}
    w_\varepsilon:=(G_\alpha*\mu)*\rho_\varepsilon.
\end{equation}
Then $w_\varepsilon$ is smooth and strictly positive on $\mathbb{R}^d$. In particular,
\begin{equation}
\label{eq:bessel-potential-lambda}
    \lambda_\varepsilon:=\inf_{x\in K}w_\varepsilon(x)>0.
\end{equation}
If $\phi\in C_0^\infty(\mathbb{R}^d)$ satisfies $\phi=1$ on a neighbourhood of $K$, then
\begin{equation}
\label{eq:bessel-admissible-function}
    u_\varepsilon:=\lambda_\varepsilon^{-1}\phi w_\varepsilon
\end{equation}
is admissible for $K$.
\end{lemma}

\begin{proof}
Since $G_\alpha$ and $\rho_\varepsilon$ are nonnegative and $\mu$ is a positive nonzero measure, we have
\[
    w_\varepsilon(x)
    =\int_{K}(G_\alpha*\rho_\varepsilon)(x-y)\,\mathrm{d}\mu(y)>0,
    \qquad x\in\mathbb{R}^d.
\]
The convolution with $\rho_\varepsilon$ makes $w_\varepsilon$ smooth. Therefore $w_\varepsilon$ is continuous, and since $K$ is compact the positive function $w_\varepsilon$ has a positive minimum on $K$. This proves \eqref{eq:bessel-potential-lambda}. If $x\in K$, then $\phi(x)=1$ and hence
\[
    u_\varepsilon(x)=\lambda_\varepsilon^{-1}w_\varepsilon(x)\geq 1.
\]
Thus $u_\varepsilon$ is admissible for $K$.
\end{proof}

\subsection{Upper bounds from MMBS data}
\label{subsec:upper-bounds-mmbs}

The preceding admissible function gives useful capacity estimates when the weight $k$ is dominated by the inverse multiplier of the Bessel kernel. We state the assumption in the form used below: there are constants $\alpha > 0$ and $A_k>0$ such that
\begin{equation}
\label{eq:bessel-weight-domination}
    k(\xi)\leq A_k(1+|\xi|^2)^{\alpha/2},
    \qquad \xi\in\mathbb{R}^d.
\end{equation}
Equivalently, with the normalization \eqref{eq:bessel-kernel-fourier},
\[
    k(\xi)\widehat{G_\alpha}(\xi)\leq A_k.
\]

\begin{theorem}[Upper bound from Frostman--MMBS data]
\label{thm:capacity-upper-bound-frostman-mmbs}
Let $K\subset\mathbb{R}^d$ be compact, and let $\mu$ be a finite positive nonzero Borel measure supported on $K$. Assume that $\mu$ has Frostman growth of exponent $0<\gamma<d$ and Fourier decay of order $0<\beta<d$ in the sense of Definitions \ref{def:frostman-growth} and \ref{def:fourier-decay-measure}. Let
\[
    p\geq p_\ast(\gamma,\beta,d)
\]
with $p_\ast$ as in \eqref{eq:mmbs-threshold}. Let $k$ be a tempered weight satisfying \eqref{eq:bessel-weight-domination}. Fix a nonnegative mollifier $\rho\in C_0^\infty(\mathbb{R}^d)$ with unit integral, $\varepsilon>0$, and a cutoff $\phi\in C_0^\infty(\mathbb{R}^d)$ such that $\phi=1$ on a neighbourhood of $K$. Then
\begin{equation}
\label{eq:capacity-upper-bound-frostman-mmbs}
    c_{p,k}(K)
    \leq
    \lambda_\varepsilon^{-1}
    C_{\phi,k}
    A_k
    C_{\mathrm{MMBS}}
    \mu(K)^{1/2},
\end{equation}
where
\[
    C_{\phi,k}:=\|M_k\widehat\phi\|_{L^1(\mathbb{R}^d)},
\]
$\lambda_\varepsilon$ is defined by \eqref{eq:bessel-potential-lambda}, and $C_{\mathrm{MMBS}}$ is the constant in \eqref{eq:mmbs-extension-estimate}.
\end{theorem}

\begin{proof}
Let $w_\varepsilon$ and $u_\varepsilon$ be defined by \eqref{eq:mollified-bessel-potential} and \eqref{eq:bessel-admissible-function}. By Lemma \ref{lem:bessel-potential-majorant}, $u_\varepsilon$ is admissible for $K$. Therefore
\[
    c_{p,k}(K)\leq \|u_\varepsilon\|_{p,k}.
\]
Using the multiplier estimate \eqref{eq:schwartz-multiplier-estimate}, we obtain
\[
    \|u_\varepsilon\|_{p,k}
    \leq
    \lambda_\varepsilon^{-1}C_{\phi,k}\|w_\varepsilon\|_{p,k}.
\]
On the Fourier side,
\[
    \widehat{w_\varepsilon}(\xi)
    =\widehat{G_\alpha}(\xi)\widehat\mu(\xi)\widehat{\rho_\varepsilon}(\xi).
\]
Since $\rho_\varepsilon$ is nonnegative with unit integral, $|\widehat{\rho_\varepsilon}(\xi)|\leq1$. Combining this with \eqref{eq:bessel-weight-domination} gives
\begin{equation}
\label{eq:bessel-potential-fourier-bound}
    |k(\xi)\widehat{w_\varepsilon}(\xi)|
    \leq
    A_k|\widehat\mu(\xi)|.
\end{equation}
Hence
\[
    \|w_\varepsilon\|_{p,k}
    \leq
    A_k\|\widehat\mu\|_{L^p(\mathbb{R}^d)}.
\]
Finally, Corollary \ref{cor:mmbs-integrability-measure} gives
\[
    \|\widehat\mu\|_{L^p(\mathbb{R}^d)}
    \leq
    C_{\mathrm{MMBS}}\,\mu(K)^{1/2}.
\]
Combining the estimates proves \eqref{eq:capacity-upper-bound-frostman-mmbs}.
\end{proof}

\subsection{Bilateral capacity estimates}
\label{subsec:bilateral-capacity-estimates}

The upper estimate above and the dual lower estimate from Proposition \ref{prop:capacity-dual-lower-bound} may be combined. The two sides need not use the same measure: the upper bound uses a measure with Frostman and Fourier-decay data, while the lower bound uses any positive measure satisfying the weighted dual integrability condition.

\begin{proposition}[Bilateral capacity estimate]
\label{prop:bilateral-capacity-estimate}
Let $K\subset\mathbb{R}^d$ be compact, let $1\leq p\leq\infty$, and let $k$ be a tempered weight. Let $\nu$ be a finite positive Borel measure supported on $K$ such that
\[
    k^{-1}(\cdot)\widehat\nu(-\cdot)\in L^{p'}(\mathbb{R}^d).
\]
Assume, in addition, that there is a measure $\mu$ supported on $K$ satisfying the hypotheses of Theorem \ref{thm:capacity-upper-bound-frostman-mmbs}. Then
\begin{equation}
\label{eq:bilateral-capacity-estimate}
    \frac{\nu(K)}
    {\|k^{-1}(\cdot)\widehat\nu(-\cdot)\|_{L^{p'}(\mathbb{R}^d)}}
    \leq
    c_{p,k}(K)
    \leq
    \lambda_\varepsilon^{-1}
    C_{\phi,k}
    A_k
    C_{\mathrm{MMBS}}\, 
    \mu(K)^{1/2},
\end{equation}
for any $p\geq p_\ast(\gamma,\beta,d)$, with $p_\ast$ as in \eqref{eq:mmbs-threshold}.The constants and auxiliary choices on the right-hand side are those of Theorem \ref{thm:capacity-upper-bound-frostman-mmbs}.
\end{proposition}

\begin{proof}
The left-hand inequality is Proposition \ref{prop:capacity-dual-lower-bound} applied to $\nu$. The right-hand inequality is Theorem \ref{thm:capacity-upper-bound-frostman-mmbs} applied to $\mu$.
\end{proof}

As a direct consequence of Proposition \ref{prop:capacity-dual-lower-bound}, the lower estimate may be optimized over all testing measures supported on $K$.

\begin{corollary}[Optimized dual lower bound]
\label{cor:optimized-dual-lower-bound}
Let $K\subset\mathbb{R}^d$ be compact, let $1\leq p\leq\infty$, and let $k$ be a tempered weight. Let
\[
    \mathcal{M}_{+,k}^{p'}(K)
    :=
    \left\{
    \nu\in\mathcal{M}_+(K):
    k^{-1}(\cdot)\widehat\nu(-\cdot)
    \in L^{p'}(\mathbb{R}^d)
    \right\},
\]
where $\mathcal{M}_+(K)$ denotes the finite positive Borel measures supported on $K$. Then
\begin{equation}
\label{eq:optimized-dual-lower-bound}
    c_{p,k}(K)
    \geq
    \sup_{\nu\in\mathcal{M}_{+,k}^{p'}(K)}
    \frac{\nu(K)}
    {\|k^{-1}(\cdot)\widehat\nu(-\cdot)\|_{L^{p'}(\mathbb{R}^d)}}.
\end{equation}
Equivalently, if
\[
    \mathcal{P}_{k}^{p'}(K)
    :=
    \left\{
    \nu\in\mathcal{M}_{+,k}^{p'}(K):\nu(K)=1
    \right\},
\]
then
\begin{equation}
\label{eq:normalized-dual-optimization}
    c_{p,k}(K)
    \geq
    \left(
    \inf_{\nu\in\mathcal{P}_{k}^{p'}(K)}
    \|k^{-1}(\cdot)\widehat\nu(-\cdot)\|_{L^{p'}(\mathbb{R}^d)}
    \right)^{-1},
\end{equation}
with the convention that $(\inf\emptyset)^{-1}=0$.
\end{corollary}

\begin{proof}
The estimate \eqref{eq:optimized-dual-lower-bound} follows from Proposition \ref{prop:capacity-dual-lower-bound} by taking the supremum over all measures in $\mathcal{M}_{+,k}^{p'}(K)$. For the normalized form, let $\nu\in\mathcal{M}_{+,k}^{p'}(K)$ be nonzero and set $\widetilde\nu:=\nu/\nu(K)$. Then $\widetilde\nu\in\mathcal{P}_{k}^{p'}(K)$ and
\[
    \|k^{-1}(\cdot)\widehat\nu(-\cdot)\|_{L^{p'}(\mathbb{R}^d)}
    =
    \nu(K)
    \|k^{-1}(\cdot)\widehat{\widetilde\nu}(-\cdot)\|_{L^{p'}(\mathbb{R}^d)}.
\]
Thus
\[
    \frac{\nu(K)}
    {\|k^{-1}(\cdot)\widehat\nu(-\cdot)\|_{L^{p'}(\mathbb{R}^d)}}
    =
    \|k^{-1}(\cdot)\widehat{\widetilde\nu}(-\cdot)\|_{L^{p'}(\mathbb{R}^d)}^{-1}.
\]
Taking the supremum over all nonzero measures in $\mathcal{M}_{+,k}^{p'}(K)$ gives \eqref{eq:normalized-dual-optimization}.
\end{proof}

\section{Weighted Lebedev criterion and examples}
\label{sec:lebedev-examples}

The lower bounds in Section \ref{subsec:testing-against-measures} may be tested with any positive measure supported on the compact set. In the present section we use the distinguished testing measure
\[
    \mathrm{d}\nu=\chi_K\,\mathrm{d}x,
\]
when $K$ has positive Lebesgue measure. In this case the dual condition becomes the weighted integrability of $\widehat{\chi_K}$. Lebedev-type estimates provide this integrability from the size of the boundary of $K$.

\subsection{Lebedev integrability for characteristic functions}
\label{subsec:lebedev-integrability}

For $1\leq q\leq\infty$, write
\begin{equation}\label{eq:ap-space-definition}
    A_q(\mathbb{R}^d)
    :=
    \{f\in\mathcal{S}'(\mathbb{R}^d):\widehat f\in L^q(\mathbb{R}^d)\}.
\end{equation}
The following form of Lebedev's argument is the one used below.

\begin{lemma}[Lebedev-type integrability]
\label{lem:lebedev-integrability}
Let $E\subset\mathbb{R}^d$ be a bounded measurable set with positive Lebesgue measure. Let
\[
    \alpha=\overline{\dim}_{M}(\partial E)<d
\]
be the upper Minkowski dimension of its boundary. If $0\leq\alpha<d$, then
\begin{equation}\label{eq:lebedev-threshold}
    \chi_E\in A_q(\mathbb{R}^d)
    \qquad\text{for every}\qquad
    q>\frac{2d}{2d-\alpha},
    \quad 1\leq q<2.
\end{equation}
Equivalently, $\widehat{\chi_E}\in L^q(\mathbb{R}^d)$ for every $q$ in the range above.
\end{lemma}

\begin{proof}
We give the argument in the form needed later. Since $E$ is bounded, $\chi_E\in L^2(\mathbb{R}^d)$. For $0<s<1$, the fractional Sobolev seminorm of $\chi_E$ is controlled by
\[
    \int_{\mathbb{R}^d}
    \frac{|E\Delta(E-h)|}{|h|^{d+2s}}\,\mathrm{d}h,
\]
where $E\Delta(E-h)$ denotes the symmetric difference. This is the standard difference characterization of $W^{s,2}(\mathbb{R}^d)$; see, for instance, \cite[Chapter 5, Section 5]{stein1970singular}.

For small $|h|$, the symmetric difference is contained in a constant multiple of the $|h|$-neighbourhood of $\partial E$. By the definition of $\overline{\dim}_{M}$, for every $\varepsilon>0$ there are constants $C_\varepsilon>0$ and $r_\varepsilon>0$ such that
\begin{equation}\label{eq:boundary-tubular-estimate}
    |\{x\in\mathbb{R}^d:\operatorname{dist}(x,\partial E)<r\}|
    \leq
    C_\varepsilon r^{d-\alpha-\varepsilon},
    \qquad 0<r<r_\varepsilon.
\end{equation}
This is the usual tubular-neighbourhood formulation of upper Minkowski dimension; see \cite[Section 3.2]{falconer2013fractal}. Hence the contribution near the origin is bounded by a constant multiple of
\[
    \int_{|h|<r_\varepsilon}
    |h|^{-d-2s}|h|^{d-\alpha-\varepsilon}\,\mathrm{d}h,
\]
which is finite whenever
\[
    2s<d-\alpha-\varepsilon.
\]
Since $\varepsilon>0$ is arbitrary, we obtain
\[
    \chi_E\in W^{s,2}(\mathbb{R}^d)
    \qquad\text{for every}\qquad
    s<\frac{d-\alpha}{2}.
\]
Equivalently,
\[
    (1+|\xi|^2)^{s/2}\widehat{\chi_E}(\xi)
    \in L^2(\mathbb{R}^d).
\]

Let $1\leq q<2$. H\"older's inequality gives
\[
\begin{aligned}
    \int_{\mathbb{R}^d}|\widehat{\chi_E}(\xi)|^q\,\mathrm{d}\xi
    &\leq
    \left(
    \int_{\mathbb{R}^d}
    |\widehat{\chi_E}(\xi)|^2(1+|\xi|^2)^s\,\mathrm{d}\xi
    \right)^{q/2}  \\
    &\quad\times
    \left(
    \int_{\mathbb{R}^d}
    (1+|\xi|^2)^{-\frac{sq}{2-q}}\,\mathrm{d}\xi
    \right)^{(2-q)/2}.
\end{aligned}
\]
The second factor is finite if and only if
\[
    \frac{2sq}{2-q}>d.
\]
Thus $\widehat{\chi_E}\in L^q(\mathbb{R}^d)$ whenever
\[
    q>\frac{2d}{d+2s}.
\]
Letting $s\uparrow(d-\alpha)/2$ gives \eqref{eq:lebedev-threshold}. This is the Lebedev mechanism in the form used here; compare \cite{lebedev2013fourier}.
\end{proof}

\subsection{A weighted Lebedev lower bound}
\label{subsec:weighted-lebedev-lower-bound}

The previous lemma can be combined with a weight that grows at infinity. This produces the weighted integrability required by the dual lower bound of Proposition \ref{prop:capacity-dual-lower-bound}.

\begin{proposition}[Weighted Lebedev lower bound]
\label{prop:weighted-lebedev-lower-bound}
Let $K\subset\mathbb{R}^d$ be a compact set with positive Lebesgue measure, and let
\[
    \alpha=\overline{\dim}_{M}(\partial K)<d.
\]
Assume that the tempered weight $k$ satisfies
\begin{equation}\label{eq:weighted-lebedev-weight-condition}
    \frac{1}{k(\xi)}
    \leq
    C_k(1+|\xi|^2)^{-\sigma/2},
    \qquad \xi\in\mathbb{R}^d,
\end{equation}
for some $\sigma\in\mathbb{R}$. If
\begin{equation}\label{eq:weighted-lebedev-nonempty-range}
    d-\alpha+2\sigma>0,
\end{equation}
then, for every $1\leq p'<2$ satisfying
\begin{equation}\label{eq:weighted-lebedev-threshold}
    p'>
    \frac{2d}{2d-\alpha+2\sigma},
\end{equation}
one has
\begin{equation}\label{eq:weighted-lebedev-integrability}
    k^{-1}\widehat{\chi_K}\in L^{p'}(\mathbb{R}^d).
\end{equation}
Consequently, for the conjugate exponent $p$,
\begin{equation}\label{eq:weighted-lebedev-capacity-lower-bound}
    c_{p,k}(K)
    \geq
    \frac{|K|}
    {\|k^{-1}\widehat{\chi_K}(-\cdot)\|_{L^{p'}(\mathbb{R}^d)}}
    >0.
\end{equation}
\end{proposition}

\begin{proof}
Fix
\[
    s_0<\frac{d-\alpha}{2}.
\]
By Lemma \ref{lem:lebedev-integrability}, or more precisely by the Sobolev estimate proved there,
\begin{equation}\label{eq:chi-sobolev-regularity}
    \int_{\mathbb{R}^d}
    |\widehat{\chi_K}(\xi)|^2(1+|\xi|^2)^{s_0}
    \,\mathrm{d}\xi
    <\infty.
\end{equation}
Using \eqref{eq:weighted-lebedev-weight-condition}, we estimate
\[
\begin{aligned}
    \int_{\mathbb{R}^d}
    |k(\xi)^{-1}\widehat{\chi_K}(\xi)|^{p'}\,\mathrm{d}\xi
    &\leq
    C_k^{p'}
    \int_{\mathbb{R}^d}
    |\widehat{\chi_K}(\xi)|^{p'}
    (1+|\xi|^2)^{-\sigma p'/2}
    \,\mathrm{d}\xi .
\end{aligned}
\]
Rewrite the integrand as
\[
    |\widehat{\chi_K}(\xi)|^{p'}
    (1+|\xi|^2)^{-\sigma p'/2}
    =
    \left(
    |\widehat{\chi_K}(\xi)|^{p'}(1+|\xi|^2)^{s_0p'/2}
    \right)
    (1+|\xi|^2)^{-(s_0+
    \sigma)p'/2}.
\]
For $1\leq p'<2$, H\"older's inequality with exponents
\[
    \frac{2}{p'}
    \qquad\text{and}\qquad
    \frac{2}{2-p'}
\]
gives
\[
\begin{aligned}
    &\int_{\mathbb{R}^d}
    |\widehat{\chi_K}(\xi)|^{p'}
    (1+|\xi|^2)^{-\sigma p'/2}
    \,\mathrm{d}\xi \\
    &\leq
    \left(
    \int_{\mathbb{R}^d}
    |\widehat{\chi_K}(\xi)|^2(1+|\xi|^2)^{s_0}
    \,\mathrm{d}\xi
    \right)^{p'/2}
    \left(
    \int_{\mathbb{R}^d}
    (1+|\xi|^2)^{-\frac{(s_0+\sigma)p'}{2-p'}}
    \,\mathrm{d}\xi
    \right)^{(2-p')/2}.
\end{aligned}
\]
The first factor is finite by \eqref{eq:chi-sobolev-regularity}. The second factor is finite precisely when
\[
    \frac{2(s_0+\sigma)p'}{2-p'}>d,
\]
or equivalently
\[
    p'>\frac{2d}{d+2s_0+2\sigma}.
\]
Letting $s_0\uparrow(d-
\alpha)/2$ gives \eqref{eq:weighted-lebedev-integrability} for every $p'$ satisfying \eqref{eq:weighted-lebedev-threshold}. The condition \eqref{eq:weighted-lebedev-nonempty-range} is exactly the condition that the threshold in \eqref{eq:weighted-lebedev-threshold} is strictly smaller than $2$, so that a nonempty range with $p'<2$ remains.

Finally apply Proposition \ref{prop:capacity-dual-lower-bound} with the positive measure
\[
    \mathrm{d}\nu=\chi_K\,\mathrm{d}x.
\]
Since $\nu(K)=|K|$ and $\widehat\nu=\widehat{\chi_K}$, \eqref{eq:weighted-lebedev-capacity-lower-bound} follows.
\end{proof}

\begin{remark}\label{rem:weighted-lebedev-effective-boundary}
The threshold in \eqref{eq:weighted-lebedev-threshold} behaves as if the boundary exponent $\alpha$ were replaced by $\alpha-2\sigma$. Positive values of $\sigma$ correspond to weights growing at infinity and improve the integrability of $k^{-1}\widehat{\chi_K}$. Negative values of $\sigma$ correspond to weights decaying at infinity and reduce the available range of exponents.
\end{remark}

\subsection{Ahlfors-David examples}
\label{subsec:ad-regular-examples}

The upper estimate of Section \ref{subsec:upper-bounds-mmbs} applies directly when a natural Hausdorff measure on $K$ has both Frostman growth and Fourier decay.

\begin{example}[Ahlfors-David upper bound]
\label{cor:ad-regular-upper-bound}
Let $K\subset\mathbb{R}^d$ be compact, and let $0<s<d$. Assume that
\[
    \mu:=\mathcal{H}^s\lfloor K
\]
is a finite positive nonzero measure and satisfies the Frostman growth condition
\[
    \mu(B(x,r))\leq C_\mu r^s,
    \qquad x\in\mathbb{R}^d,\quad r>0.
\]
Assume also that $\mu$ has Fourier decay of order $0<\beta<d$ in the sense of Definition \ref{def:fourier-decay-measure}. Let
\[
    p\geq p_\ast(s,\beta,d)
\]
with $p_\ast$ defined in \eqref{eq:mmbs-threshold}. If $k$ satisfies the Bessel domination hypothesis \eqref{eq:bessel-weight-domination}, then the upper bound \eqref{eq:capacity-upper-bound-frostman-mmbs} holds with $\gamma=s$ and this choice of $\mu$.
\end{example}

\begin{proof}
The hypotheses place $\mu=\mathcal{H}^s\lfloor K$ within the scope of Theorem \ref{thm:capacity-upper-bound-frostman-mmbs}, with $\gamma=s$. Applying that theorem gives the stated estimate.
\end{proof}
Although $\lambda = \inf_{x \in K} w(x) > 0$, highly sparse regions in $K$ can cause $\lambda \to 0$, leading the upper bound to diverge. Bounding $\lambda$ uniformly away from zero requires geometric thickness conditions on $K$, such as Ahlfors-David (AD) regularity. Indeed, suppose $K$ is $\gamma$-dimensional AD-regular with constant $A > 0$. Fixing $0 < r \le \operatorname{diam}(K)$ and noting that $G_\alpha$ is strictly decreasing and radial, we obtain for any $x \in K$:$$(G_\alpha \ast \mu)(x) \geq \left( \inf_{y \in B(x,r)} G_\alpha(x-y) \right) \mu(B(x,r)) \geq G_\alpha(z_r) \frac{r^\gamma}{A} > 0,$$where $\vert{}z_r\vert{} = r$. Since convolution with the positive mollifier $\rho_\epsilon$ preserves this uniform lower bound, $\lambda$ is bounded away from zero independently of $x$. We defer a full exploration of these geometric conditions to future research.

\subsection{Positive-measure Cantor examples}
\label{subsec:fat-cantor-examples}

The weighted Lebedev criterion is especially useful for compact sets of positive measure whose boundary is large (for example, sets whose interior is not dense at any point) but still has upper Minkowski dimension strictly below $d$.

\begin{example}[Positive-measure Cantor sets]
\label{ex:fat-cantor-lebedev}
Let $K\subset\mathbb{R}^d$ be a compact set with $|K|>0$ and
\[
    \alpha=\overline{\dim}_{M}(\partial K)<d.
\]
Then Proposition \ref{prop:weighted-lebedev-lower-bound} gives a positive lower bound for $c_{p,k}(K)$ for every tempered weight satisfying \eqref{eq:weighted-lebedev-weight-condition} and every conjugate exponent $p'$ in the range \eqref{eq:weighted-lebedev-threshold}.

This applies, in particular, to positive-measure Cantor-type sets for which the boundary Minkowski dimension can be estimated. For instance, the family of one-dimensional fat Cantor sets obtained from the Cantor function by the Liu--Pego construction \cite{Liu02072024} gives compact sets of positive measure; taking products gives higher-dimensional examples. Once an upper Minkowski bound for the boundary is available, the criterion above yields the corresponding weighted Fourier-capacity lower bound.
\end{example}

\bibliographystyle{amsplain}
\bibliography{ref}
\end{document}